\documentclass[reqno]{amsart}
\usepackage{amssymb}
\address{School of Mathematical Sciences, Beijing Normal University,
Laboratory of Mathematics
 and Complex Systems,  Ministry of
  Education,
   Beijing 100875, The People's Republic
 of China.}\email{gclu@bnu.edu.cn}

\address{College of  Science, China University of Petroleum-Beijing , Beijing 102249,
 The People's Republic
 of China.} \email{xmchen@cup.edu.cn}

\usepackage{amsmath,amssymb,latexsym}

\usepackage{mathrsfs} 
\usepackage{amssymb,latexsym}
\usepackage{amsmath,latexsym}
\usepackage{amscd,latexsym} 
\usepackage[all]{xy}

\newcommand{\R}{{\mathbb R}}

\newcommand{\Z}{{\mathbb Z}}

\newtheorem{theorem}{Theorem}[section]

\newtheorem{remark}[theorem]{Remark}

\newtheorem{lemma}[theorem]{Lemma}

\newtheorem{proposition}[theorem]{Proposition}
\newtheorem{claim}[theorem]{Claim}

\numberwithin{equation}{section}

\begin{document}

\title[Deformations of special Legendrian submanifolds]
{Deformations of special Legendrian submanifolds with boundary (corrected version)}

\author[Guangcun Lu]{Guangcun Lu}
\thanks{\endgraf
Partially supported by the NNSF 10971014 and 11271044 of China, PCSIRT, RFDPHEC (No.
200800270003) and the Fundamental Research Funds for the Central Universities (No. 2012CXQT09).}

\author{Xiaomin Chen}

\begin{abstract}
This is a corrected version of our paper published in Osaka Journal of Mathematics
 51(2014), 673-693. We correct Theorem~1.1, Proposition~3.3 and their proofs.
\end{abstract}

\date{May 11, 2015}


 \maketitle

 \noindent{\bf Preface} Georgios Dimitroglou Rizell, in his review MR3272612 in MathSciNet, pointed
out ``The main result Theorem 1.1 claims that such submanifolds constitute a discrete set, which
however seems to be incorrect. A counterexample is provided by Example~2.7 contained
in the same paper: in this case all structures, including $W$, are invariant under the
one-parameter Reeb flow. The image under the Reeb flow of the special Legendrian
submanifold in the same example is thus a one-parameter family of special Legendrian
submanifolds with boundary on $W$ satisfying the required properties. The mistake seems
to originate from Proposition~3.3, which is false." The aim
of this version is to correct  the original Theorem~1.1 and Proposition~\ref{prop:3.3}
and their proofs. We change completely the content of the original Remark~3.5 and delete
few sentences, for example, one below Claim~2.6 in the previous version.
We also correct few typo error and polish few sentences.

\section{Introduction and main results}

The calibrated geometry was invented by Harvey and Lawson in their
seminal paper \cite{HaLa}. A class of important calibrated
submanifolds is special Lagrangian submanifolds in Calabi-Yau
manifolds. Let $(M,J,\omega,\Omega)$ be a real $2n$-dimensional
Calabi-Yau manifold. A special Lagrangian submanifold of it is a
$n$-dimensional
submanifold $L$ with $\omega|_L=0$ and ${\rm Im}(\Omega)|_L=0$. In
1996 McLean \cite{McL} developed the deformation theory of special
Lagrangian submanifolds (and other special calibrated submanifolds)
and showed:\vspace{2mm}

\noindent{\bf McLean theorem} (\cite{McL}). {\it A normal vector
field $V$ to a compact special Lagrangian submanifold $L$ without
boundary in $(M,J,\omega,\Omega)$ is the deformation vector field to
a normal deformation through special Lagrangian submanifolds if and
only if the corresponding $1$-form $(JV)^\flat$ on $L$ is harmonic.
There are no obstructions to extending a first order deformation to
an actual deformation and the tangent space to such deformations can
be identified through the cohomology class of the harmonic form with
$H^1(L;\mathbb{R})$.}\vspace{2mm}

Since then  the theory is generalized to various situations. See
\cite{Joy, Joy1, Pa} and references therein. For example, S. Salur
\cite{Sa} generalized McLean theorem to symplectic manifolds. We
here only list those closely related to ours. The first one is the
case of compact special Lagrangian submanifolds with nonempty
boundary considered by Butsher \cite{Bu}. He called a submanifold
$L$ in the Calabi-Yau manifold $(M,J,\omega,\Omega)$ {\it minimal
Lagrangian} if $\omega|_L=0$ and ${\rm Im}(e^{i\theta}\Omega)|_L=0$
for some $\theta\in\mathbb{R}$. If $L$ is a Lagrangian submanifold
of $(M, \omega)$ with nonempty boundary $\partial L$ and
$N\in\Gamma(T_{\partial L}L)$ is the inward unit normal vector field
of $\partial L$ in $L$,  he  defined a {\it scaffold} for $L$ to be
a submanifold $W$ of $M$ such that $\partial L\subset W$, the bundle
$(TW)^\omega$ is trivial, and that $N$ is a smooth section of the
bundle $(T_{\partial L}W)^\omega$. \vspace{2mm}

\noindent{\bf Butsher theorem} (\cite{Bu}). {\it Let $L$ be a
special Lagrangian submanifold of a compact Calabi-Yau manifold $M$
with non-empty boundary $\partial L$ and let $W$ be a symplectic,
codimension two scaffold for $L$. Then the space of minimal
Lagrangian submanifolds sufficiently near $L$ (in a suitable
$C^{1,\beta}$ sense ) but with boundary on $W$ is finite dimensional
and is parametrized over the harmonic 1-forms of $L$ satisfying
Neumann boundary conditions.}\vspace{2mm}

The work inspired Kovalev and Lotay \cite{KoLo} to study the
analogous deformation problem of a compact coassociative $4$-fold
with boundary inside a particular fixed $6$-dimensional submanifold  with a compatible
Hermitian symplectic structure in a $7$-manifold with
closed $G_2$-structures. Recently Gayet and Witt \cite{GaWi} also
investigated the deformation of a compact associative submanifold
with boundary in a coassociative submanifold in a topological
$G_2$-manifold.

As a natural generalization of the Calabi-Yau manifolds in the
context of contact geometry  Tomassini and Vezzoni
\cite[Definition 3.1]{ToVe} introduced the notion of a contact Calabi-Yau
manifold, cf. Definition~2.1.  Let $(M,\eta, J,\epsilon)$ be a
$(2n+1)$-dimensional contact Calabi-Yau manifold, and
$j:L\hookrightarrow M$ be a compact special Legendrian submanifold
without boundary (cf. Definition~2.2). Two special Legendrian
submanifolds $j_0: L\hookrightarrow M$ and $j_1: L\hookrightarrow M$
are called {\it deformation equivalent} if there exists a smooth map
$F:
L\times[0,1]\to M$ such that\\
$\bullet$ $F(\cdot,t)\colon L\times \{t\}\to M$ is a special
Legendrian embedding
  for any $t\in[0,1]$;\\
$\bullet$ $F(\cdot,0)=j_0$, $F(\cdot,1)=j_1$.\\
(cf.\cite[Definition 4.4]{ToVe}). If there exists a diffeomorphism
$\phi\in{\rm Diff}(L)$ such that $j_1=j_0\circ\phi$ we say $j_0$ and
$j_1$ to be {\it equivalent}. This yields an equivalent relation
$\sim$ among all embeddings from $L$ to $M$. Let
$\widetilde{\mathfrak{M}}(L)$ be the set of special Legendrian
submanifolds of $(M,\alpha,J,\epsilon)$
  which are deformation equivalent to $j\colon L \hookrightarrow M$.
Call $\mathfrak{M}(L):=\widetilde{\mathfrak{M}}(L)/\sim$
 the {\it moduli space of special Legendrian  submanifolds} which are
deformation equivalent to $j: L\hookrightarrow M$.\vspace{2mm}
Tomassini and Vezzoni \cite[Theorem 4.5]{ToVe}  proved: \vspace{2mm}

\noindent{\bf Tomassini-Vezzoni theorem}(\cite[Theorem 4.5]{ToVe}).
\quad{\it Let $(M,\eta, J,\epsilon)$ be a contact Calabi-Yau
manifold of dimension $2n+1$, and $L\subset M$ be a compact
special Legendrian submanifold without boundary. Then the moduli
space $\mathfrak{M}(L)$ is a smooth one-dimensional manifold.}
\vspace{2mm}

Motivated by the above works,  we study in this paper the local
deformations of compact special Legendrian submanifolds with
(nonempty) boundary. (The boundary  is always assumed to be smooth
throughout this paper.) Different from the case $\partial
L=\emptyset$ considered by Tomassini and Vezzoni \cite{ToVe}, it is
showed in Remark~5.1 that the moduli space $\mathfrak{M}(L)$ is
infinite dimensional.

In order to get interesting results it is necessary to add some
boundary conditions. Inspired by  \cite[Definition 1]{Bu} we introduce a
notion of {\it scaffold} for $L$ in Definition~2.3, which is a
suitable contact submanifold $W$. Two special Legendrian
submanifolds $j_0: L\hookrightarrow M$ and $j_1: L\hookrightarrow M$
with $j_0(\partial L)\subset W$ and $j_1(\partial L)\subset W$ are
called {\it deformation equivalent} if there exists a smooth map $F:
L\times[0,1]\to M$ such that\\
$\bullet$ $F(\cdot,t)\colon L\times \{t\}\to M$ is a special
Legendrian embedding with $F(\partial L, t)\subset W$
  for any $t\in[0,1]$;\\
$\bullet$ $F(\cdot,0)=j_0$, $F(\cdot,1)=j_1$.\\

 The {\it moduli space of special Legendrian  submanifolds} which are
deformation equivalent to $j\colon L\hookrightarrow M$ with
$j(\partial L)\subset W$ is defined as
$$
\begin{array}{l}
\mathfrak{M}(L, W):=\bigl\{\mbox{\rm special
Legendrian submanifolds of}\,\, (M,\alpha,J,\epsilon)\\[5pt]
\hskip3truecm\mbox{\rm which are deformation equivalent to}\,\,
j\colon L \hookrightarrow M\\[5pt]
\hskip3truecm\hbox{with $j(\partial L)\subset W$ and are near
$j$}\bigr\}/\sim.
\end{array}
$$
Denote by $g$ the Riemannian metric $\frac{1}{2}d\alpha(\cdot,J\cdot)+\alpha\otimes\alpha$
on $M$, see (\ref{e:2.1}) for precise constructions. Let $N(L)$ be the normal bundle of $L$ with respect to $g$,
and let $\Gamma(N(L))_W$ be the set of all $V\in\Gamma(N(L))$
 that are the deformation vector fields to normal deformation through special Lagrangian submanifolds
with boundary confined in $W$.
 Our first result is

\begin{theorem}\label{th:1.1}
Let $(M,J,\alpha,\epsilon)$ be a contact Calabi-Yau manifold, and
$L$ be a connected compact special Legendrian submanifold with nonempty
boundary $\partial L$ inside a  scaffold $W$ of codimension two.
Then the following statements hold:\\
{\rm (i)} $\mathfrak{M}(L, W)$ has at most dimension $\dim H^1(L;\mathbb{R})+1$ near $L$;
moreover $\Gamma(N(L))_W$ is a vector space of dimension at most
 $\dim H^1(L;\mathbb{R})+1$.\\
{\rm (ii)}  $\{V\in\Gamma(N(L))_W\,|\,d\bigl(\alpha({V})|_{\partial L}\bigr)=0\}$
 has dimension at most $l$, and one if $\partial L$ is connected,
 where $l$ is the number of connected components of $\partial L$;
 moreover $\dim\{V\in\Gamma(N(L))_W\,|\,\alpha({V})|_{\partial L}={\rm const}\}=1$. \\
{\rm (iii)} Any  vector field $V\in\Gamma(T_LM)\setminus\Gamma(TL)$  with  $\alpha(V)|_{\partial L}=0$
cannot be the deformation vector field to a deformation through special Lagrangian submanifolds
with boundary confined in $W$.
 \end{theorem}


(i) is similar to the above Butsher theorem.
In (ii), the second statement and the first one in case $l=1$ are
  similar to the case of compact special Legendrian
submanifolds  without boundary considered in Tomassini-Vezzoni
Theorem above.  The local rigidity in (iii) is similar to the case of a compact simply connected
 special Lagrangian submanifold  without boundary in  McLean Theorem, and Simons'
rigidity result of stable minimal submanifolds with fixed boundary in \cite{Sim}.

Now we turn to consider weaker boundary conditions. Let $(M,\alpha,
J,\epsilon)$ be a $(2n+1)$-dimensional contact Calabi-Yau manifold,
and $L\subset M$ be a compact special Legendrian submanifold with
(non-empty) boundary. A normal vector field $V$ to $L$ is called
{\it boundary $\alpha$-constant} if $\alpha(V)|_{\partial L}$ is
constant.
 The following result, which is stated in a similar way to McLean
 Theorem above,  is similar to that of Tomassini and Vezzoni
 \cite{ToVe}.

\begin{theorem}\label{th:1.2}
Let $(M,\alpha, J,\epsilon)$ be a $(2n+1)$-dimensional contact
Calabi-Yau manifold, and and $L\subset M$ be a compact special
Legendrian submanifold with (non-empty) boundary. A boundary
$\alpha$-constant normal vector field $V$ to $L$ is the deformation
vector field to a normal deformation through special Legendrian
submanifolds if and only if $\alpha(V)$ is constant. Moreover the
tangent space to such deformations is given by $\mathbb{R}R_\alpha$,
where $R_\alpha$ is the Reeb vector field of $\alpha$.
\end{theorem}

 Similar to the case $L$ being compact and without boundary considered in Theorem~4.5
 of \cite{ToVe}  the deformation in Theorem~\ref{th:1.2} is also given
by the isometries generated by the Reeb vector field, which is
completely different from the  deformation without boundary
constraints as proved in Remark~5.1.

 The key points in the proofs of
Theorems~\ref{th:1.1} and \ref{th:1.2} are to find a suitable
definition of scaffold for a special Legendrian submanifold with
boundary and to prove a corresponding result  with Lemma 5 of
\cite{Bu}, Lemma~\ref{lm:3.1}. For the former we propose and study
it in Section 2. The proof of the latter will be given in Section 3
and  is more troublesome because we need to use not only contact
neighborhood theorem but also symplectic neighborhood theorem. In
Sections 4 and 5, we complete the proofs of Theorems~\ref{th:1.1}
and \ref{th:1.2}  respectively.\vspace{2mm}

\section{Preliminaries}\label{sec:2}

\subsection{Contact Calabi-Yau manifolds and special Legendrian
submanifolds}
 Let $(M,\alpha)$ be a contact manifold with contact distribution
 $\xi=\ker\alpha$ and  Reeb vector field  $R_\alpha$.
Then $\kappa:=d\alpha/2$ restricts to a symplectic vector bundle
structure on $\xi\to M$, $\kappa|_\xi$, and
 every compatible complex structure $J\in{\mathcal
J}(\xi, \kappa|_{\xi})$ gives a Riemannian metric $g_J$ on the
bundle $\xi\to M$, $g_J(u,v)=\kappa(u,Jv)$ for $u,v\in\xi$. By
setting $J(R_\alpha)=0$ we can extend $J$ to an endomorphism of
$TM$,  {\it  also denoted by} $J$  without special statements.
Clearly
\begin{equation}\label{e:2.1}
J^2=-\textbf{I}+ \alpha\otimes R_\alpha, \quad\hbox{and}\quad g :=
g_J + \alpha\otimes\alpha
\end{equation}
is a Riemannian metric $g$ on $M$, where \textbf{I} is the identity
endomorphism on $TM$. Define a Nijenhuis tensor of $J$ by
$$
N_{J}(X,Y)=[J X,J Y]-J[X,J Y]-J[J X, Y] + J^2[X,Y]
$$
for all $X,Y\in TM$. If $N_{J}
  =- d\alpha\otimes R_{\alpha}$ then the pair $(\alpha, J)$ is a
{\it Sasakian structure} on $M$, and the triple $(M,\alpha, J)$ is
called a  {\it Sasakian manifold}. On such a manifold it holds that
$d\Lambda^r_B(M)\subset\Lambda^r_B(M)$ and $J\bigl(
\Lambda^r_B(M)\bigr)=\Lambda^r_B(M)$, where $\Lambda_B^r(M)$ is the
set of all differential $r$-form $\gamma$ on $M$ with $
\iota_{R_\alpha}\gamma=0$ and $\mathcal{L}_{R_\alpha}\gamma=0$. So
we have a split
$$
\Lambda^{r}_B(M)\otimes\mathbb{C}=\oplus_{p+q=r}\Lambda^{p,q}_J(\xi)
$$
and $\kappa=(1/2)d\alpha\in \Lambda^{1,1}_J(\xi)$.\vspace{2mm}

\noindent{\rm Definition 2.1}(\cite[Definition 2.1]{ToVe}).\quad  A {\it
contact Calabi-Yau manifold} is a quadruple $(M,\alpha, J,\epsilon)$
consisting of a $2n+1$-dimensional Sasakian manifold $(M,\alpha, J)$
and a nowhere vanishing basic form
 $\epsilon\in \Lambda_{J}^{n,0}({\xi})$  such that
$$
    \epsilon\wedge \bar{\epsilon}=c_{n}\frac{\kappa^{n}}{n!}
   $$
and
$$
d\epsilon=0,
$$
where $c_{n}=(-1)^{\frac{n(n+1)}{2}}(2i)^{n}$ and
$\kappa=(1/2)d\alpha$.\vspace{2mm}

\noindent{\rm Definition 2.2} (\cite[Definition 4.2]{ToVe}).\quad  Let
$(M^{2n+1},\alpha, J,\epsilon)$ be a contact Calabi-Yau manifold. An
embedding $p:L\to M$ is called a {\it special Legendrian
submanifold} if $\dim L=n$, $p^{*}\alpha=0$ and $p^{*}{\rm
Im}\epsilon=0$.\vspace{2mm}

Clearly, $p^\ast\epsilon=p^\ast({\rm Re}\epsilon)$ is a volume form
on $L$. Thus every special Legendrian submanifold has a natural
orientation. By \cite[p.722]{McL} or \cite[Proposition 2.6]{DeB} we have
\begin{equation}\label{e:2.2}
p^\ast\left(\iota_Y{\rm Im}\epsilon\right)=-\star\left(
p^\ast(\iota_Y\kappa)\right)=-\frac{1}{2}\star\left(
p^\ast(\iota_Yd\alpha)\right)
\end{equation}
for any section $Y:L\to p^\ast\xi$, where the star operator $\star$
is computed with respect to $p^*(g_{J})=p^\ast(\kappa\circ({\rm
id}\times J))$ and the volume form ${\rm
Vol}(L):=p^\ast\epsilon=p^*({\rm Re}\epsilon)$.

For any $n$-dimensional manifold $N$,  the cotangent bundle $T^\ast
N$ has a canonical $1$-form $\lambda_{\rm can}$. The $1$-jet bundle
$J^1N=\mathbb{R}\times T^\ast N$ is a contact manifold with contact
form $\alpha=\pi_1^\ast(dt)-\pi_2^\ast(\lambda_{\rm can})$ and Reeb
vector field $\partial/\partial t$, where $t\in\mathbb{R}$ is the
real parameter and $\pi_i$ is the projection from $\mathbb{R}\times
T^\ast N$ onto the $i$-th factor, $i=1,2$. (See \cite[Example
3.44]{McSa}).

\subsection{Boundary conditions}

Corresponding to  \cite[Def. 1]{Bu} we introduce:\vspace{2mm}

\noindent{\rm Definition 2.3.}\quad  Let $L$ be a submanifold of the
contact manifold $(M,\xi=\ker\alpha)$ with boundary $\partial L$ and
let $N\in \Gamma(T_{\partial L}L)$ be the inward unit normal vector
field of $\partial L$ in $L$.  A contact submanifold $(W,\xi')$ of
$(M,\xi)$ is called a {\it scaffold} for $L$ if
\begin{enumerate}
\item[(i)] $\partial
L\subset W$,
\item[(ii)]  $N\in\Gamma(\xi'^{\bot}|_{\partial L})$, and
\item[(iii)]
       the bundle $\xi'^\perp$ is trivial, where $\xi'^\bot$ is
       the symplectically orthogonal complement of
$\xi'$ in $(\xi|_W, \kappa|_{\xi|_W})$.
\end{enumerate}
\vspace{2mm}

Given a contact manifold $(M,\alpha)$ let $J$ and $g$ be as in
(\ref{e:2.1}).   If $(W, \xi')$ is a contact submanifold of $(M,
\xi=\ker\alpha)$, that is, $T_xW\cap\xi_x=\xi_x'$ for all $x\in W$,
the following claim shows that the condition (iii) of Definition~2.3
is equivalent to one that $(TW)^{\bot_g}$ is trivial, where
$(TW)^{\bot_g}$ denotes the orthogonal complementary bundle of $TW$
in $T_WM$ with respect to the
 metric $g$.
\vspace{2mm}

\noindent{\bf Claim 2.4.}\quad {\it
$(TW)^{\bot_g}=(J\xi')^\bot=J(\xi'^\bot)$.}\vspace{2mm}

\noindent{\rm Proof}.\quad For $x\in W$, since
$\xi_x'^\bot\subset\xi_x$ and $J_x$ restricts to a complex structure
on $\xi_x$ we have
\begin{eqnarray*}
\xi_x'^\bot&=&\{v\in\xi_x\,|\,\kappa(v,u)=0\;\forall
u\in\xi_x'\}\\
&=&\{v\in\xi_x\,|\,\kappa(Jv,Ju)=0\;\forall
u\in\xi_x'\}\\
&=&\{v\in\xi_x\,|\, g_J(Jv,u)=0\;\forall
u\in\xi_x'\}\\
&=&\{v\in\xi_x\,|\, g(Jv, bR_\alpha+ u)=0\;\forall\;
bR_\alpha+ u\in R_\alpha\mathbb{R}+\xi_x'\}\\
&=&\{v\in\xi_x\,|\, g(Jv, Y)=0\;\forall Y\in T_xW\}.
\end{eqnarray*}
This implies $J\xi'^\bot=(TW)^{\bot_g}$ or
$\xi'^\bot=J(TW)^{\bot_g}$. Moreover, both $J\xi'^\bot$ and
$\xi'^\bot$ are contained in $\xi|_W$, and $\xi$ is $J$-invariant.
It is easy to check that $J\xi'^\bot=(J\xi')^\bot$.
\hfill$\Box$\vspace{2mm}

\vspace{2mm} \noindent{\bf Proposition~2.5.}\quad{\it  Let $L$ be a
Legendrian submanifold of the contact manifold $(M,\xi=\ker\alpha)$
with (nonempty) boundary $\partial L$ and let $W$ be a scaffold for
$L$. Then $\partial L$ is a Legendrian submanifold of $(W, \xi')$.}
\vspace{2mm}

\noindent{\rm Proof}.\quad  Since $L$ is the Legendrian submanifold
of $(M,\xi)$, $TL\subset\xi|_L$. Moreover the definition of the
scaffold implies that $T\partial L\subset T_{\partial L}W$ and thus
$T\partial L\subset T_{\partial L}W\cap\xi|_{\partial
L}=\xi'|_{\partial L}$. This shows that  the boundary $\partial L$
is a Legendrian submanifold of $(W,\xi')$.\hfill$\Box$\vspace{2mm}

Under the assumptions of Proposition~2.5, let $f_t:L\to M$ be a
deformation of $L$ satisfying $f_t(\partial L)\subset W$ for all
$t$, and let $V=\frac{d}{dt}f_t|_{t=0}$ be the corresponding
deformation vector field. Clearly, $V(x)\in T_xW$ for any
$x\in\partial L$. Since $L$ is a Legendrian submanifold, we have
$TL\subset\xi|_L$.
 Note that  $N(x)\in T_xL$ for any $x\in\partial
L$. Then the condition (ii) of Definition~2.3 implies that
$N(x)\in\xi_x'^\bot$, and so ${N}(x)\in T_xL\cap\xi'^\bot_x$ and
$$
J_xN(x)\in J_x(T_xL\cap\xi'^\bot_x)\subset J_x\xi'^\bot_x\subset
J_x\xi_x=\xi_x.
$$
Since $W$ is a contact submanifold, we may write $V(x)=Y+
aR_\alpha(x)$, where $Y\in\xi_x'$. By Claim~2.4, $J_x{N}(x)\in
J_x\xi'^\bot_x=(T_xW)^{\bot_g}$ and thus
$$
0=g(J_x{N}(x), V(x))=g_J(J_x{N}(x), Y)=\kappa(J_x{N}(x),
J_xY)=\kappa({N}(x), Y).
$$
Note that $Y=V(x)-\alpha(V(x))R_\alpha(x)$ and that
$\iota_{R_\alpha}d\alpha=0$. We get\vspace{2mm}

\noindent{\bf Claim~2.6.}\quad{\it If $f_t:L\to M$ be a deformation
of $L$ satisfying $f_t(\partial L)\subset W$ for all $t$, then the
corresponding deformation vector
field $V$ satisfies {\rm Neumann boundary condition:} $d\alpha({N}(x),
V(x))=0\;\forall x\in\partial L$.}\vspace{2mm}


\noindent{\bf Eexample~2.7.}\quad It is not hard to construct an
example satisfying the boundary conditions of Theorems~1.1 and 1.2.
Let $(x_1,\cdots,x_n,y_1,\cdots,y_n, z)$ denote the standard
Euclidean coordinate in $\R^{2n+1}$.  The standard contact
Calabi-Yau structure $(\alpha, J, \epsilon)$ on $\R^{2n+1}$ is given
by
$$
\alpha=2dz-2\sum^n_{j=1}y_jdx_j,\quad \epsilon=(dx_1+
idy_1)\wedge\cdots\wedge(dx_n+ idy_n)
$$
and
$$
 J:\xi={\rm Ker}(\alpha)={\rm
Span}(\{y_1\partial_z+\partial_{x_1},\cdots, y_n\partial_z+
\partial_{x_n}, \partial_{y_1},\cdots,\partial_{y_n}\})\to\xi
$$
where $J$ is given by $JX_r=Y_r=\partial_{y_r}$ and
$JY_r=-X_r=-y_r\partial_z-\partial_{x_r}$, $r=1,\cdots,n$. (See
\cite[Example 3.2]{ToVe}). Observe that this structure is invariant under
the action of the subgroup $\Z^n\times\{0\}^{n+1}$ of $\Z^{2n+1}$.
It descends to such a structure on
$M=\R^{2n+1}/(\Z^n\times\{0\}^{n+1})= {\R^n/\Z^n}\times\R^{n+1}$,
also denoted by $(\alpha, J, \epsilon)$ without occurs of
confusions. As usual we write the point of $M$ as
$([x_1],\cdots,[x_n],y_1,\cdots,y_n,z)$. Let $n\ge 2$. Consider the
contact submanifold of $(M, \alpha)$, $W=W_{0}\cup W_{1}$,
$$
W_{k}=\left\{([x_1],\cdots,[x_n],y_1,\cdots,y_{n-1},0, z)\in M\,\biggm|\,
x_n=\frac{k+1}{3}\right\},\quad k=0, 1.
$$
Since the contact form on it is
$\alpha'=\alpha|_W=2dz-2\sum^{n-1}_{j=1}y_jdx_j$, it is easy to see
that  the symplectically orthogonal complementary bundle $\xi'^\bot$ of
$\xi'={\rm Ker}(\alpha')$ in $(\xi|_W, \kappa|_{\xi|_W})$ is
trivial. In fact, we have
\begin{eqnarray*}
&&\xi'={\rm Span}(\{y_1\partial_z+\partial_{x_1},\cdots,
y_{n-1}\partial_z+
\partial_{x_{n-1}}, \partial_{y_1},\cdots,\partial_{y_{n-1}}\}),\\
&&\xi'^\bot={\rm Span}(\{y_n\partial_z+
\partial_{x_n}, \partial_{y_n}\}).
\end{eqnarray*}
Consider $L=\{([x_1],\cdots,[x_n],0,\cdots,0)\in M\,|\, 1/3\le
x_n\le 2/3\}$. It is a compact Legendrian submanifold with boundary
$\partial L=\partial_0L\cup\partial_1L$, where
$$
\partial_k L=\{([x_1],\cdots,[x_n], 0,\cdots,0)\in
M\,|\, x_n=(k+1)/3\},\quad k=0,1.
$$
Clearly, $\partial_kL\subset W_k$, $k=0,1$, and thus $\partial
L\subset W$. By (\ref{e:2.1}) the metric $g=g_J+\alpha\otimes\alpha$
satisfies: $g(R_\alpha, R_\alpha)=1$, $g(X_r, X_s)=g(Y_r,
Y_s)=\delta_{rs}$ and $g(X_r, Y_s)=g(X_r, R_\alpha)=g(Y_r,
R_\alpha)=0$ for $r,s=1,\cdots,n$. For $p=([x_1],\cdots,[x_n],
0,\cdots,0)\in\partial_0L$ we have
$$
T_pL={\rm Span}(\{\partial_{x_1}|_p,\cdots,
\partial_{x_n}|_p\}),\quad T_p\partial_0L={\rm
Span}(\{\partial_{x_1}|_p,\cdots, \partial_{x_{n-1}}|_p\}).
$$
Since $X_j|_p=\partial_{x_j}|_p$, $j=1,\cdots,n$, it follows that
$X_n|_p$ is the inward unit normal vector at $p$ of $\partial L$ in
$L$. Similarly,  for $p=([x_1],\cdots,[x_n],
0,\cdots,0)\in\partial_1L$ the inward unit normal vector at $p$ of
$\partial L$ in $L$ is $-X_n|_p$. Namely the inward unit normal
vector field $N$ of $\partial L$ in $L$ belongs to
$\Gamma(\xi'^\bot|_{\partial L})$. Hence $W$ is a scaffold for $L$.
\vspace{2mm}

\section{Constructing a  new metric}

 In the study of the deformation of the
special Legendrian submanifold $L$ without boundary by  Tomassini
and Vezzoni \cite{ToVe}, the deformations of $L$ are parameterized
by
 sections of the normal bundle $N(L)$ using the exponent map $\exp(V):L\to
M$. However, in our case, since $W$ is generally not totally
geodesic,  it cannot be assured that the image of $\partial L$ under
$\exp(V)$ sits in $W$.  In order to fix out the problem we shall
follow the ideas in \cite{Bu} to construct a new metric $\hat{g}$
such that the image of $\partial L$ under the corresponding exponent
map is contained in $W$, that is, such that $W$ is totally geodesic
near $\partial L$. The following  is an analogue
 of \cite[Lemma 5]{Bu}.

\begin{lemma}\label{lm:3.1}
Let $L$ be a compact Legendrian submanifold of the contact manifold
$(M,J,\alpha)$ with (nonempty) boundary $\partial L$ and let $W$ be
a scaffold for it of codimension two. Then there is a neighborhood
${\mathcal U}={\mathcal U}(\partial L, M)$ of $\partial L$ in $M$
and a contact embedding
 $\phi:{\mathcal U}\to \mathbb{R}\times
T^{*}(\partial L)\times\mathbb{R}^2$ such that the following
conditions hold:
\begin{enumerate}
\item[(i)] $\phi\bigr(W\cap{\mathcal U}\bigl)\subset \mathbb{R}\times T^{*}(\partial
L)\times\{(0,0)\}$,
\item[(ii)] $\phi(\partial L)=\{0\}\times\partial
L\times\{0,0\}$,

\item[(iii)] $(t, x, v, s_1, s_2)\in\phi({\mathcal U})\Rightarrow
(t, x, v, 0, 0)\in\phi({\mathcal U})$,

\item[(iv)] for any nowhere zero smooth section $V: W\to\xi'^\bot|_{W}$,
  $\phi$ can be required to satisfy
  $\phi_*(V(p))=\frac{\partial}{\partial s_1}\Bigm|_{\phi(p)}$ for any $p\in\partial L$,
  where $(s_1,s_2)$ are the coordinate
functions of $\mathbb{R}^2$.
\end{enumerate}
\end{lemma}

 Note that the condition (iv) is slightly weaker than the
corresponding one of \cite[Lemma 5(4)]{Bu}. It is sufficient for us
to construct a suitable metric in Proposition~\ref{prop:3.2}. Even
so our proof uses not only contact neighborhood theorem but also
symplectic neighborhood theorem  in contrast with the proof of
\cite[Lemma 5(4)]{Bu}. It is a key of our proof.\vspace{2mm}

\noindent{\bf Proof of Lemma~3.1}.\quad Since $\partial L$ is a
compact Legendrian submanifold of $\mathbb{R}\times T^{*}(\partial
L)$ without boundary, from the Neighborhood Theorem for Legendrian
(cf. Corollary~2.5.9 in \cite{Ge}) it follows that there exists a
contactomorphism $\phi_0$ from a neighborhood ${\mathcal
U}_0(\partial L, W)$ of $\partial L$ in $W$ to one ${\mathcal
V}_0(0_{\partial L})$ of the zero section of $T^{*}(\partial L)$ in
$\mathbb{R}\times T^{*}(\partial L)$ such that
\begin{equation}\label{e:3.1}
\phi_0(x)=(0, x)\quad\forall x\in\partial L.
\end{equation}

Fix a Riemannian metric on the bundle $T^{*}(\partial L)$, and then
take a sufficiently small $\epsilon>0$ such that
\begin{equation}\label{e:3.2}
M_1':=\{(t,x, v)\,:\, |t|\le\epsilon,\;v\in T_x^{*}(\partial
L)\;\hbox{with}\;|v|\le\epsilon\,\}\subset {\mathcal
V}_0(0_{\partial L}).
\end{equation}
We get another neighborhood of $\partial L$ in $W$,
\begin{equation}\label{e:3.3}
M_0':=\phi_0^{-1}(M_1')\subset {\mathcal U}_0(\partial L, W)\subset
W.
\end{equation}
Then $\phi_0:M_0'\to M_1'$ is a contactomorphism. Obverse that
$M_0'$ and $M_1'$ are compact contact submanifolds of $W$ and
$T^{*}(\partial L)\times\mathbb{R}$ with boundary and of codimension
zero, respectively.

Let $\lambda_{\rm can}$ denote the canonical $1$-form on
$T^\ast\partial L$. Recall that the contact form and Reeb vector
field on $J^1{\partial L}=\mathbb{R}\times T^{*}(\partial L)$ are
\begin{equation}\label{e:3.4}
\tilde\beta=dt-\lambda_{\rm can}\qquad\hbox{and}\qquad
R_{\tilde\beta}=\frac{\partial}{\partial t}.
\end{equation}
 Assume that $s_1,s_2$ are the coordinate functions of
$\mathbb{R}^2$. We have a contact form on $J^1{\partial
L}\times\mathbb{R}^2=\mathbb{R}\times T^{*}(\partial
L)\times\mathbb{R}^2$,
\begin{equation}\label{e:3.5}
\beta=\tilde\beta-s_1ds_2=dt-\lambda_{\rm can}-s_1ds_2,
\end{equation}
whose Reeb vector field is given by $R_\beta=\partial/\partial t$.
Denote by $(\ker(\tilde\beta))^\bot$  the symplectically orthogonal
complement of $\ker(\tilde\beta)$ in
 $\ker(\beta)$ (with respect to $d\beta$). It is easily checked that it is equal to the
 trivial bundle
 $$
{\rm Span}\biggl(\Bigl\{\frac{\partial}{\partial{s_1}},
\frac{\partial}{\partial{s_2}}\Bigl\}\biggr)\to J^1{\partial
L}\times\mathbb{R}^2.
 $$

Define  $M_0:=M$, $M_1:=\mathbb{R}\times T^{*}(\partial
L)\times\mathbb{R}^2$, and $M_0'$ and $M_1'$ as above. (Identify
$M_1'\equiv M_1'\times\{(0,0)\}\subset J^1L\times\mathbb{R}^2$).
Since $\xi'^\perp$ is trivial we can pick two vector fields
$V_1,V_2$ such that $V_1,V_2$ form a basis of $\xi'^\perp$ and
satisfy $d\alpha(V_1,V_2)=0$. There exists an obvious symplectic
vector bundle isomorphism
$$
\xi'^\perp|_{M_0'}\to {\rm
Span}\biggl(\Bigl\{\frac{\partial}{\partial{s_1}},
\frac{\partial}{\partial{s_2}}\Bigl\}\biggr)\Biggm|_{M_1'}
$$
given by
$$
\Phi(V_1(x))=\frac{\partial}{\partial{s_1}}\Bigm|_{(\phi_0(x),0,0)}
$$
and
$$
\Phi(V_2(x))=\frac{\partial}{\partial{s_2}}\Bigm|_{(\phi_0(x),0,0)}
$$
for any $x\in M_0'$. By Theorem~2.5.15 of \cite{Ge}, we may extend
$\phi_0$ into a contactomorphism $\phi_1$ from a neighborhood
${\mathcal U}(M_0')$ of $M_0'$ in $M_0=M$ to that ${\mathcal
U}(M_1')$ of $M_1'\equiv M_1'\times\bigl\{(0,0)\}$ in $M_1$ such
that $T\phi_1|_{\xi'^\perp|_{M_0'}}$ and $\Phi$ are bundle homotopic
(as symplectic bundle isomorphisms) up to a conformality. ({\it
Note}: From the proof of \cite[Th.2.5.15]{Ge} it is not hard to see
that the theorem still holds if compact contact submanifold $M_i'$
have boundary and $M_i'\subset{\rm Int}(M_i)$.)

Actually, we may assume that ${\mathcal U}(M_1')$ has the following
form:
\begin{eqnarray}\label{e:3.6}
{\mathcal U}(M_1')&=&\{(t,x,v)\,:\,|t|<\varepsilon,\; v\in
T_x^{*}(\partial L)\;\hbox{with}\;|v|<\varepsilon' \}\nonumber\\
& &\times\{(s_1, s_2)\in\mathbb{R}^2\,:\, |s_1|, |s_2|<\delta\},
\end{eqnarray}
where $0<\varepsilon'<\varepsilon$ and $\delta>0$, and
$${\mathcal
U}(M_0'):=\phi_1^{-1}({\mathcal U}(M_1')).
$$

By suitably shrinking ${\mathcal U}(M_0')$ and ${\mathcal U}(M_1')$
if necessary, we can require
\begin{eqnarray}
&&W_0:=W\cap{\mathcal U}(M_0')\subset {\mathcal U}_0(\partial L,
W),\label{e:3.7}\\
&& \phi_1(W_0)\subset  \mathbb{R}\times T^\ast\partial
L\times\{(0,0)\},\label{e:3.8}\\
&&(t,x,v, s_1, s_2)\in {\mathcal U}(M_1')\Longrightarrow
\phi_1^{-1}(t,x,v,0,0)\in W_0.\nonumber
\end{eqnarray}
Clearly, ${\mathcal U}(M_0')$ and $\phi_1$ satisfy the conditions
(i)-(iii) in Lemma~\ref{lm:3.1}.

For (iv) we need to modify $\phi_1$ and ${\mathcal U}(M_0')$. Since
$\phi_1$ is a contactomorphism,
$$
\phi_{1\ast}\bigl( \xi'^\bot|_{W_0} \bigr)\subset {\rm
Span}\biggl(\Bigl\{\frac{\partial}{\partial{s_1}},
\frac{\partial}{\partial{s_2}}\Bigl\}\biggr)\biggm|_{\phi_1(W_0)}.
$$
It follows that  there exist smooth real functions $f_1,
f_2:W_0\to\mathbb{R}$ such that
$$
\phi_{1*}(V(x))=f_1(x)\frac{\partial}{\partial s_1}|_{\phi(x)} +
f_2(x)\frac{\partial}{\partial s_2}|_{\phi(x)}
$$
and
$$
|f_1(x)|+ |f_2(x)|\ne 0
$$
for any $x\in W_0$, where  $V: W\to\xi'^\bot|_{W}$ is the given
nowhere zero smooth section in Lemma~\ref{lm:3.1}(iv).

Take $\epsilon>0$ sufficiently small so that
$$
 R_\epsilon:=\{(t, x,
v,0,0)\in\mathbb{R}\times T^\ast \partial L\times\mathbb{R}^2\,:\,
|t|\le\epsilon,\;|v|\le\epsilon\}\subset\phi_1(W_0).
$$
Consider the compact symplectic submanifold of $\bigl(T^\ast\partial
L\times\mathbb{R}^2, -d\lambda_{\rm can}-ds_1\wedge ds_2\bigr)$,
\begin{equation}\label{e:3.9}
S_\epsilon:=\{(x, v,0,0)\in T^\ast\partial L\times\mathbb{R}^2\,:\,
|v|\le\epsilon\}.
\end{equation}
 Its symplectic normal bundle is
$$
{\rm Span}\biggl(\Bigl\{\frac{\partial}{\partial{s_1}},
\frac{\partial}{\partial{s_2}}\Bigl\}\biggr)\biggm|_{S_\epsilon},
$$
and $\phi_{1\ast}(V)$ restricts to a nowhere zero smooth section
\begin{equation}\label{e:3.10}
p\mapsto
f_1\circ\phi^{-1}(p)\frac{\partial}{\partial s_1}\Bigm|_{p} +
f_2\circ\phi^{-1}(p)\frac{\partial}{\partial s_2}\Bigm|_{p}.
\end{equation}
Obverse that there exists an obvious symplectic vector bundle
isomorphism
$$
\Psi:{\rm Span}\biggl(\Bigl\{\frac{\partial}{\partial{s_1}},
\frac{\partial}{\partial{s_2}}\Bigl\}\biggr)\biggm|_{S_\epsilon}\to
{\rm Span}\biggl(\Bigl\{\frac{\partial}{\partial{s_1}},
\frac{\partial}{\partial{s_2}}\Bigl\}\biggr)\biggm|_{S_\epsilon}
$$
which sends the section in (\ref{e:3.10}) to one
$$
p\mapsto\frac{\partial}{\partial s_1}\Bigm|_{p}.
$$
Hence the symplectic neighborhood theorem \footnote{From the proof
of \cite[Theorem 3.30]{McSa} it is not hard to see that the theorem still
holds if compact symplectic submanifold $Q_j$ have boundary and
$Q_j\subset{\rm Int}(M_j)$. } (cf. \cite[Theorem 3.30]{McSa}) yields a
symplectomorphism between neighborhoods of $S_\epsilon$ in
$\bigl(T^\ast\partial L\times\mathbb{R}^2, -d\lambda_{\rm
can}-ds_1\wedge ds_2\bigr)$,
$$
\varphi: {\mathcal N}_0(S_\epsilon)\to {\mathcal N}_1(S_\epsilon)
$$
such that
\begin{equation}\label{e:3.11}
\varphi(p)=p\quad\hbox{and}\quad d\varphi(p)=\Psi_p
\end{equation}
for any $p\in S_\epsilon$. In particular, we have
\begin{equation}\label{e:3.12}
 d\varphi(p)\Bigl(
\phi_{1\ast}(V)|_p \Bigr)=\frac{\partial}{\partial
s_1}\Bigm|_{p}\qquad\forall p\in S_\epsilon.
\end{equation}
Since (\ref{e:3.5}) implies
$$
\ker(\beta)\bigm|_{(t,x,v,s_1,s_2)}=T_{(x,v)}T^\ast\partial
L\times{\rm Span}\biggl(\Bigl\{\frac{\partial}{\partial{s_1}},
\frac{\partial}{\partial{s_2}}\Bigl\}\biggr)\biggm|_{(s_1,s_2)},
$$
the map
\begin{equation}\label{e:3.13}
\phi_2:\mathbb{R}\times
{\mathcal N}_0(S_\epsilon)\to \mathbb{R}\times {\mathcal
N}_1(S_\epsilon),\;(t,p)\mapsto (t,\varphi(p))
\end{equation}
must be a contactomorphism with respect to the induced contact
structure from $(\mathbb{R}\times T^{*}(\partial
L)\times\mathbb{R}^2,\beta)$.

Take a  neighborhood ${\mathcal U}$ of $\partial L$ in $M$ such that
\begin{eqnarray*}
&&{\mathcal U}\subset {\mathcal
U}(M_0')\quad\hbox{and}\quad\phi_1({\mathcal U})\subset
\mathbb{R}\times {\mathcal N}_0(S_\epsilon),\\
&&(t, x, v, s_1, s_2)\in\phi_2(\phi_1({\mathcal U}))\Longrightarrow
(t, x, v, 0, 0)\in\phi_2(\phi_1({\mathcal U})).
\end{eqnarray*}
 Then
the composition $\phi:=\phi_2\circ(\phi_1|_{\mathcal U})$ is a
contact embedding from ${\mathcal U}$ into $(\mathbb{R}\times
T^{*}(\partial L)\times\mathbb{R}^2,\beta)$ such that the condition
(iii) is satisfied. By (\ref{e:3.8}) and (\ref{e:3.11}) it is easy
to see that (i) is satisfied for $\phi$ and ${\mathcal U}$, i.e.
$$
 \phi(W\cap{\mathcal U})\subset
\mathbb{R}\times T^\ast\partial L\times\{(0,0)\}.
$$
From (\ref{e:3.1}) and (\ref{e:3.11}) it follows that $\phi(\partial
L)=\{0\}\times\partial L\times\{0,0\}$. That is, (i) holds. Finally,
(\ref{e:3.12}) implies that $\phi$ satisfies the condition (iv),
i.e.
$$
d\phi(p)(V(p))=\frac{\partial}{\partial
s_1}\Bigm|_{\phi(p)}\quad\forall p\in\partial L.
$$
\hfill$\Box$\vspace{2mm}

\newpage

As in \cite{Bu}, with Lemma~3.1  we may construct the desired metric
$\hat{g}$ as follows.

\textbf{Step 1}. Recall that $N$ is  the inward unit normal vector field
of $\partial L$ in $L$ and $N\in\Gamma(\xi'^\bot|_{\partial L})$.
Let ${\mathcal U}$ and $\phi$ be as in the Lemma \ref{lm:3.1} with
$\phi_*(N(p))=\frac{\partial}{\partial s_1}\Bigm|_{\phi(p)}$ for any
$p\in\partial L$. By shrinking $W$ we assume that $N$ has been
extended into a nowhere zero section in $\Gamma(\xi'^\bot|_W)$.
Hence using Lemma~\ref{lm:3.1}(iii) we may define a metric $g'$ on
$\phi({\mathcal U})$  as follows:
$$
g'(t,x,v,
s_1,s_2):=(\phi^{-1})^*\bigr(g|_W(\phi^{-1}(t,x,v,0,0))\bigl)+ds_1\otimes
ds_1+ds_2\otimes ds_2
$$
for every $(t,x,v,s_1,s_2)\in \phi({\mathcal U})$.

 \textbf{Step 2}.
Consider the metric $g_1:=\phi^\ast g'$ on ${\mathcal U}$. Take a
neighborhood ${\mathcal V}$ of $\partial L$ in $M$ such that the
closure of ${\mathcal V}$ is contained in ${\mathcal U}$.
 Let $\rho:M\to\mathbb{R}$ be a smooth function such
that $\rho=1$ on a neighborhood ${\mathcal V}$, and $\rho=0$ outside
${\mathcal U}$. We then define the metric $\hat{g}$ by
$$
\hat{g}:=\rho g_1+(1-\rho)g.
$$

The following two propositions correspond to Propositions~6 and 7 in
\cite{Bu}, respectively.

\begin{proposition}\label{prop:3.2}
For the neighborhood ${\mathcal V}$ of $\partial L$ in Step 2,
$W\cap {\mathcal V}$ is totally geodesic with respect to the metric
$\hat{g}$.
\end{proposition}

\noindent{\bf Proof.}\quad For any $p\in W\cap {\mathcal V}$,
let $\phi(p)=(t(p),\phi^\ast(p), 0,0)\in
 \mathbb{R}\times T^{*}(\partial
L)\times\{(0,0)\}$, where $\phi^\ast(p)\in T^\ast L$. By composing
the map $\phi$ in Lemma~\ref{lm:3.1} with the canonical coordinate system on $T^\ast L$
around $\phi^\ast(p)$ we obtain a local contact coordinate system around
it,
$$
\mathcal
{O}(p)\to\mathbb{R}\times\mathbb{R}^{2n-2}\times\mathbb{R}^2,\;
q\mapsto (t(q), z_1(q),\cdots, z_{2n-2}(q), s_1(q), s_2(q))
$$
such that
\begin{enumerate}
\item[$\bullet$] for some
smooth function $h:{\mathcal O}(p)\to\mathbb{R}$ it holds that
\begin{equation}\label{e:3.14}
\alpha|_{\mathcal
{O}(p)}=e^h\left(dt-\sum^{n-1}_{k=1}z_{n-1+k}dz_k-s_2ds_1\right),
\end{equation}
  and the Reeb field $R_\alpha=\frac{\partial}{\partial t}$;

\item[$\bullet$] $W\cap{\mathcal O}(p)\ni q\mapsto (t(q), z_1(q),\cdots,
z_{2n-2}(q))$  is
 a local contact coordinate system around $p$ in  the relatively open neighborhood
$W\cap {\mathcal O}(p)$  and
\begin{equation}\label{e:3.15}
\alpha|_{W\cap {\mathcal
O}(p)}=e^{h_0}\left(dt-\sum^{n-1}_{k=1}z_{n-1+k}dz_k\right),
\end{equation}
where $h_0=h|_{W\cap {\mathcal O}(p)}$. Moreover the Reeb field of
$\alpha|_{W\cap {\mathcal O}(p)}$ is given by the restriction of
$\frac{\partial}{\partial t}$ to $W\cap {\mathcal O}(p)$.
\end{enumerate}

For convenience we write $t$ as $z_0$. In the corresponding local
coordinate vector fields
$$
\frac{\partial}{\partial z_0}=\frac{\partial}{\partial t},
\frac{\partial}{\partial z_1}, \frac{\partial}{\partial z_2},\cdots,
\frac{\partial}{\partial z_{2n-2}},\frac{\partial}{\partial s_1},
\frac{\partial}{\partial s_2}
$$
 we have
\begin{equation}\label{e:3.16}
\hat{g}=\sum^{n-1}_{k, l=0}(g|_{W})_{kl}dz_k\otimes dz_l+ds_1\otimes ds_1+ ds_2\otimes ds_2.
\end{equation}
It is easily computed that
$$
\hat{g}\left(\nabla_{\frac{\partial}{\partial
z_k}}\frac{\partial}{\partial z_l},
  \frac{\partial}{\partial s_i}\right)=\frac{1}{2}
  \bigl(\hat g_{z_ks_i,z_l}+ \hat g_{z_ls_i,z_k}-\hat g_{z_kz_l,s_i}\bigr)=0.
  $$
So the second fundamental form of $W\cap {\mathcal V}$ with respect
to $\hat g$ vanishes, that is, $W\cap {\mathcal V}$ is totally
geodesic.\hfill$\Box$\vspace{2mm}

\begin{proposition}\label{prop:3.3}
Let $L$ be a compact Legendrian submanifold with boundary of the
contact manifold $(M,\alpha)$, and let $W$ be a codimension two
scaffold for $L$. Denote by $\widehat{N}(L)$ the normal bundle of
$L$ with respect to $\hat{g}$. For $p\in\partial L$, suppose that
$\hat{V}\in\widehat{N}_p(L)$ satisfies the boundary condition
$$
(d\alpha)_p(N(p), \hat V)=0.
$$
Then  $\hat{V}\in T_pW$,  and  $\hat{V}-\alpha(\hat{V})R_\alpha(p)$ cannot be in $T_pL$ if it is not zero.
\end{proposition}

\noindent{\bf Proof.}\quad For any point $p\in\partial L$, take the
local coordinate system around it on $\partial L$,
${\bf U}(p)\ni q\to (x_1(q),\cdots,x_{n-1}(q))\in\mathbb{R}^{n-1}$, such that
$x_j(p)=0$, $j=1,\cdots,n-1$. It induce a natural
local coordinate system around $p=(p,0)$ on $T^\ast(\partial L)$,
$$
\pi^{-1}({\bf U}(p))\ni (q,v)\mapsto (x_1(q),\cdots,x_{n-1}(q), y_1(q,v),\cdots,y_{n-1}(q,v)),
$$
where $\pi:T^\ast(\partial L)\to\partial L$ is the natural bundle projection.

Let $\phi$ be as in Lemma~\ref{lm:3.1} and satisfy
  $\phi_*(N(p))=\frac{\partial}{\partial s_1}\bigm|_{\phi(p)}$ for any $p\in\partial L$.
Choose a small open neighborhood $\mathcal
{O}(p)$ of $p$ in $M$ such that $\mathcal
{O}(p)\subset \mathcal{U}(p)$ and
$$
\phi(\mathcal
{O}(p))\cap\bigl(\{0\}\times T^{*}(\partial
L)\times\{(0,0)\}\bigr)\subset\{0\}\times\pi^{-1}({\bf U}(p))\times\{0\}\equiv \pi^{-1}({\bf U}(p)),
$$
where  $\mathcal{U}(p)$ and $\phi$ are as in Lemma~\ref{lm:3.1}.
For $q\in \mathcal
{O}(p)$ let $\phi(q)=(t(q), \tilde{q}, s_1(q), s_2(q))$, where $\tilde{q}\in T^{*}(\partial
L)$, and $z_j(q)=x_j(\pi(\tilde{q}))$ and $z_{n-1+j}(q)=y_j(\tilde{q})$ for $j=1,\cdots,n-1$.
Then $t_1(p)=s_1(p)=s_2(p)=z_j(p)=0$, $j=1,\cdots,2n-2$ and
$$
\mathcal
{O}(p)\to\mathbb{R}\times\mathbb{R}^{2n-2}\times\mathbb{R}^2,\;
q\mapsto (t(q), z_1(q),\cdots, z_{2n-2}(q), s_1(q), s_2(q))
$$
is a coordinate system satisfying (\ref{e:3.14})-(\ref{e:3.15}). Moreover
 we have
 \begin{enumerate}
\item[(A)]  $\frac{\partial}{\partial
s_1}\bigm|_p$ and $\frac{\partial}{\partial
s_2}\bigm|_p$ are $\hat{g}_p$-orthogonal, and they are also $\hat{g}_p$-orthogonal to $T_pW$;

\item[(B)] $\frac{\partial}{\partial z_1}\bigm|_p,\cdots,
\frac{\partial}{\partial z_{n-1}}\bigm|_p$
forms a basis of $T_p\partial L$,  and
$\frac{\partial}{\partial z_1}\bigm|_p,\cdots,
\frac{\partial}{\partial z_{n-1}}\bigm|_p, \frac{\partial}{\partial
s_1}\bigm|_p$
is a basis of $T_pL$ since the normal vector field $N$ of $\partial L$ in $L$ in the local coordinate system is
equal to $\frac{\partial}{\partial s_1}$;

\item[(C)] $\xi'_p$ is spanned by $\frac{\partial}{\partial z_{n}}\bigm|_p,\cdots, \frac{\partial}{\partial z_{2n-2}}\bigm|_p$ and
$$
\left(z_n\frac{\partial}{\partial t}+ \frac{\partial}{\partial z_{1}}\right)\Bigm|_p=
\frac{\partial}{\partial z_{1}}\Bigm|_p,\cdots,
\left(z_{2n-2}\frac{\partial}{\partial t}+ \frac{\partial}{\partial z_{n-1}}\right)\Bigm|_p=
\frac{\partial}{\partial z_{n-1}}\Bigm|_p;
$$
\item[(D)] $\xi_p$ is spanned by $\frac{\partial}{\partial z_{n}}\bigm|_p,\cdots, \frac{\partial}{\partial z_{2n-2}}\bigm|_p$, $\frac{\partial}{\partial s_{2}}\bigm|_p$ and
   $\left(s_2\frac{\partial}{\partial t}+ \frac{\partial}{\partial s_{1}}\right)\Bigm|_p=
\frac{\partial}{\partial s_{1}}\Bigm|_p$,
$$
\left(z_n\frac{\partial}{\partial t}+ \frac{\partial}{\partial z_{1}}\right)\Bigm|_p=
\frac{\partial}{\partial z_{1}}\Bigm|_p,\cdots,
\left(z_{2n-2}\frac{\partial}{\partial t}+ \frac{\partial}{\partial z_{n-1}}\right)\Bigm|_p=
\frac{\partial}{\partial z_{n-1}}\Bigm|_p;
$$
\item[(E)] $R_\alpha(p)=\frac{\partial}{\partial t}\bigm|_p$ is $\hat{g}_p$-orthogonal to $\xi'_p$ and
$\xi_p$.
\end{enumerate}
From these we deduce that every $\hat V\in\widehat{{N}}_p(L)$ can be expressed as
$$
\hat V=a_n\frac{\partial}{\partial z_n}\Bigm|_p+\cdots +
a_{2n-2}\frac{\partial}{\partial z_{2n-2}}\Bigm|_p + b\frac{\partial}{\partial s_2}\Bigm|_p+
\lambda\frac{\partial}{\partial t}\Bigm|_p,
$$
where $b$, $\lambda$ and $a_n,\cdots,a_{2n-2}$ are real numbers.
Suppose that $(d\alpha)_p(N(p), \hat V)=0$. Since
$N(p)=\frac{\partial}{\partial s_1}\bigm|_p$,
$R_\alpha(p)=\frac{\partial}{\partial t}\bigm|_p$ and hence $i_{R_\alpha}(d\alpha)=0$ we have
\begin{equation}\label{e:3.17}
(d\alpha)_p\biggl(\frac{\partial}{\partial s_1}\Bigm|_p,
a_n\frac{\partial}{\partial z_n}\Bigm|_p+\cdots +
a_{2n-2}\frac{\partial}{\partial z_{2n-2}}\Bigm|_p + b
\frac{\partial}{\partial s_2}\Bigm|_p\biggr)=0.
\end{equation}
By (\ref{e:3.14}) it is easy computed that
\begin{eqnarray}\label{e:3.18}
d\alpha|_{\mathcal
{O}(p)}&=&e^h\left(-ds_2\wedge ds_1-\sum^{n-1}_{k=1}dz_{n-1+k}\wedge dz_k\right)\nonumber\\
&+&e^hdh\wedge\left(dt-\sum^{n-1}_{k=1}z_{n-1+k}dz_k-s_2ds_1\right).
\end{eqnarray}
 It follows from
(\ref{e:3.17})-(\ref{e:3.18}) that
\begin{eqnarray*}
0&=&-\left|\begin{array}{ll}
0& b\\
1&0
\end{array}\right|-\sum^{n-1}_{k=1}\left|\begin{array}{ll}
0&  a_{n-1+k}\\
0&0
\end{array}\right|\\
&+&\left|\begin{array}{ll}
\frac{\partial h}{\partial s_1}(p)& a_n\frac{\partial h}{\partial z_n}(p)+\cdots +
a_{2n-2}\frac{\partial h}{\partial z_{2n-2}}(p) + b
\frac{\partial h}{\partial s_2}(p)\\
0& 0
\end{array}\right|\\
&-&\sum^{n-1}_{k=1}
\left|\begin{array}{ll}
\frac{\partial h}{\partial s_1}(p)&
a_n\frac{\partial h}{\partial z_n}(p)+\cdots +
a_{2n-2}\frac{\partial h}{\partial z_{2n-2}}(p) + b
\frac{\partial h}{\partial s_2}(p)\\
0& 0\end{array}\right|\\
&-&
\left|\begin{array}{ll}
\frac{\partial h}{\partial s_1}(p)&
a_n\frac{\partial h}{\partial z_n}(p)+\cdots +
a_{2n-2}\frac{\partial h}{\partial z_{2n-2}}(p) + b
\frac{\partial h}{\partial s_2}(p)\\
s_2(p)&0
\end{array}\right|.
\end{eqnarray*}
Noting $s_2(p)=0$, we obtain $b=0$ and so
$$
\hat V=a_n\frac{\partial}{\partial z_n}\Bigm|_p+\cdots +
a_{2n-2}\frac{\partial}{\partial z_{2n-2}}\Bigm|_p +
\lambda\frac{\partial}{\partial t}\Bigm|_p\in T_pW.
$$
Clearly, if $\hat{V}-\alpha(\hat{V})R_\alpha(p)\ne 0$, it is not in $T_pL$.
 \hfill$\Box$\vspace{2mm}

\begin{remark}\label{rm:3.4}
 Let $\widehat{\exp}$ be the
exponent map of the metric $\hat{g}$. For any $p\in\partial L$ and
$v\in \widehat{N}_p(L)$ with $d\alpha(N(p), v)=0$,
Proposition~\ref{prop:3.2} and Proposition~\ref{prop:3.3}  show that
$\widehat{\exp}(p,v)\in W\cap {\mathcal U}$ if $|v|$ is small
enough.
\end{remark}

\begin{remark}\label{rm:3.5}
From (E) in the proof of Proposition~\ref{prop:3.3}
we see that $R_\alpha(p)$ is $\hat{g}_p$-orthogonal to $\xi'_p$ and $\xi_p$ at each $p\in\partial L$.
It follows that the Reeb vector $R_\alpha(p)$ at each $p\in\partial L$
belongs to not only $N_p(L)\cap\hat{N}_p(L)$ but also $N_p(\partial L)\cap\hat{N}_p(\partial L)$.
Note that we cannot obtain such conclusions at $p\in L\setminus\partial L$.
\end{remark}

\section{The proof of Theorem~\ref{th:1.1}}\label{sec:3}

\noindent{\bf 4.1. A  brief review of notations in Hodge theory}.
 For $k\in\mathbb{N}\cup\{0\}$, $1\le p<\infty$ and $0<a<1$, let
$W^{k,p}\Omega^r(L)$ (resp. $C^{k,a}\Omega^r(L)$) denote the space
of $r$-forms of class $W^{k,p}$ (resp. $C^{k,a}$) as usual (cf.
\cite{Mor, Sch}).  Each form $\omega$ of them  has a ``tangential
component" $\textbf{t}\omega$ and a ``normal component"
$\textbf{n}\omega$ (cf.
 \cite[Def.4.2]{Mor0} or \cite[(2.25)]{Sch}), which satisfy
\begin{equation}\label{e:4.1}
{\bf t}(\star\omega)=\star({\bf n}\omega)\quad\hbox{and}\quad {\bf
n}(\star\omega)=\star({\bf t}\omega)
\end{equation}
by Lemma 4.2 of \cite{Mor0}, where $\star$ is the Hodge
 star operator of the metric $\hat g$. Set
 \begin{eqnarray*}
&&C^{k,a}\Omega_{\bf D}^r(L):=\{\omega\in
C^{k,a}\Omega^r(L):\textbf{t}\omega=0\},\\
&&C^{k,a}\Omega_{\bf N}^r(L):=\{\omega\in
C^{k,a}\Omega^r(L):\textbf{n}\omega=0\}
\end{eqnarray*}
and
\begin{eqnarray*}
&&\mathcal{H}C^{k,a}\Omega^r(L):=\{\omega\in
C^{k,a}\Omega^r(L):d\omega=\delta\omega=0\},\\
&&\mathcal{H}C^{k,a}\Omega_{\bf D}^r(L):=\{\omega\in
C^{k,a}\Omega_{\bf D}^r(L):d\omega=\delta\omega=0\},\\
&&\mathcal{H}C^{k,a}\Omega_{\bf N}^r(L):=\{\omega\in
C^{k,a}\Omega_{\bf N}^r(L):d\omega=\delta\omega=0\}.
\end{eqnarray*}

Replacing $C^{k,a}$ by $W^{k,p}$ gives corresponding spaces
$W^{k,p}\Omega_{\bf D}^r(L)$, $W^{k,p}\Omega_{\bf D}^r(L)$ and
$\mathcal{H}W^{k,p}\Omega^r(L)$, $\mathcal{H}W^{k,p}\Omega_{\bf D}^r(L)$,
$\mathcal{H}W^{k,p}\Omega_{\bf N}^r(L)$.
 Clearly, for $S^r_{\bf N}=C^{k,a}\Omega^{r}_{\bf N}(L)$ and $S^r_{\bf
D}=C^{k,a}\Omega^{r}_{\bf D}(L)$ (or $S^r_{\bf N}=W^{k,
p}\Omega^{r}_{\bf N}(L)$ and $S^r_{\bf D}=W^{k, p}\Omega^{r}_{\bf
D}(L)$),  (\ref{e:4.1}) implies
\begin{equation}\label{e:4.2}
\star\bigl(S^{r}_{\bf N}\bigr)\subset S^{n-r}_{\bf
D}\quad\hbox{and}\quad \star\bigl(S^{r}_{\bf D}\bigr)\subset
S^{n-r}_{\bf N}.
\end{equation}
By the definition of the co-differential $\delta$, for any $r$-form
$\omega$ it holds that
\begin{equation}\label{e:4.3}
\star(\star\omega)=(-1)^{r(n-r)}\omega,\qquad
\star\delta\omega=(-1)^rd\star\omega,\qquad\star
d\omega=(-1)^{r+1}\delta\star\omega.
\end{equation}
 For $k\in\mathbb{N}\cup\{0\}$  the closure  $C^{k,a}(d\Omega^{r}(L))$ of
$d\Omega^r(L)$ in $C^{k,a}\Omega^{r+1}(L)$ is contained $\{d\eta:\,
\eta\in C^{k+1,a}\Omega^{r}(L)\}$ by  the Poincar\'e lemma (cf. \S3.1
of \cite{Bu}{\rm ).

\noindent{\bf 4.2. Defining the differential operator}.
By Propositions~\ref{prop:3.2},~\ref{prop:3.3} (and Tubular Neighborhood Theorem) the
sufficiently small neighborhood of the zero section of
$\widehat{{N}}(L)$ satisfying the Neumann boundary condition corresponds to
the deformations of submanifold $L$ with boundary $\partial L$
confined in $W$ in one-to-one way.


Let $C^{2,a}(\Gamma(\widehat{{N}}(L)))$ denote the Banach space of
$C^{2,a}$-sections of the bundle $\widehat{{N}}(L)$.
 Define the Banach spaces
\begin{eqnarray*}
&&{\mathcal{X}}_1:=\bigl\{V\in C^{2,a}(\Gamma(\widehat{{N}}(L)))\,:
\;d\alpha(N, V|_{\partial L})=0\;\&\;\alpha(V)|_{\partial L}=0\bigr\},\\
&&{\mathcal{X}}_2:=\bigl\{V\in C^{2,a}(\Gamma(\widehat{{N}}(L)))\,:
\;d\alpha(N, V|_{\partial L})=0\;\&\; d\alpha(V)|_{\partial L}=0\bigr\},\\
&&{\mathcal{X}}_3:=\bigl\{V\in C^{2,a}(\Gamma(\widehat{{N}}(L)))\,:
\;d\alpha(N, V|_{\partial L})=0\;\&\;\alpha(V)|_{\partial L}={\rm constants}\bigr\},\\
&&{\mathcal{X}}_4:=\bigl\{V\in C^{2,a}(\Gamma(\widehat{{N}}(L)))\,:
\;d\alpha(N, V|_{\partial L})=0\bigr\}.
\end{eqnarray*}
Then ${\mathcal{X}}_4\subset {\mathcal{X}}_3\subset {\mathcal{X}}_2\subset {\mathcal{X}}_1$.
Let ${\mathcal{X}}$ be one of ${\mathcal{X}}_i$, $i=1,2,3,4$.
Denote by  $\mathscr{U}$ a neighborhood of $0$ in
 ${\mathcal X}$.
For $V\in \mathscr{U}$ define  $\widehat\exp_V:L\to M,\,x\mapsto
\widehat\exp_V(x):=\widehat\exp_x(V(x))$. Set
\begin{eqnarray}\label{e:4.4}
 &&F:\mathscr{U}\to C^{1,a}\Omega^{1}(L)\oplus
C^{0,a}\Omega^{n}(L),\\
&&\qquad V\mapsto\bigl( (\widehat\exp_V)^\ast\alpha,
2(\widehat\exp_V)^\ast{\rm Im}\epsilon\bigr).\nonumber
\end{eqnarray}
 It is $C^1$ as done in \cite{Bu, ToVe}.
Clearly,  $\widehat\exp_{V}$ is homotopic to the inclusion
$j:L\hookrightarrow M$ via $\widehat\exp_{tV}$, and hence they
induce the same homomorphisms between the de Rham cohomology groups.
It follows that the de Rham cohomology classes
$$
[\widehat\exp_{ V}^\ast({\rm Im}\epsilon)]=\widehat\exp_{
V}^\ast[{\rm Im}\epsilon]=j^\ast[{\rm Im}\epsilon]=[j^\ast({\rm
Im}\epsilon)]\in H^n(L,\mathbb{R})\quad\hbox{vanish}.
$$
This shows that
$$
{\rm Im} F\subseteq C^{1,a}\Omega^1(L)\oplus
dC^{1,a}\Omega^{n-1}(L).
$$
Consider $F$ as a map to $C^{1,a}\Omega^1(L)\oplus
dC^{1,a}\Omega^{n-1}(L)$.

\noindent{\bf 4.3. Proving that the differential of  $F$ at $0$ is surjective}.
To compute the differential of  $F$ at $0$, for $V\in\mathcal{X}$ we
set $f=\alpha(V)$ and $Y:=V-fR_\alpha$. Then $f\in C^{2, a}(L)$ and
$Y\in C^{1,a}(\Gamma(\xi|_L))$.
 Now $V=fR_\alpha+ Y$. By the Cartan formula one can
compute the linearization of $F$ at $0$,
\begin{equation}\label{e:4.5}
   \begin{split}
     F'(0)(V)=&\frac{d}{dt}(\widehat\exp^{*}_{t V}\alpha,2\widehat\exp^{*}_{t V}{\rm Im}\epsilon)|_{t=0}\\
        =&(\mathcal{L}_{ V}\alpha,2\mathcal{L}_{ V}{\rm Im}\epsilon)|_{L}\\
        =&(d\iota_{fR_{\alpha}+ Y}\alpha+\iota_{fR_{\alpha}+ Y}d\alpha,2d
                  \iota_{fR_{\alpha}+ Y}{\rm Im}\epsilon)|_{L}\\
        =&(j^\ast df+ j^\ast\iota_{Y}d\alpha,2dj^\ast\iota_{Y}{\rm Im}\epsilon)\\
        =&(j^\ast df+ j^\ast\iota_{Y}d\alpha,-d\star j^\ast\iota_{Y}d\alpha)\\
        =&\bigl(d(f\circ j)+ j^\ast(\iota_{Y}d\alpha),-d\star
        j^\ast(\iota_{Y}d\alpha)\bigr).
     \end{split}
\end{equation}
Here the fifth equality comes from (\ref{e:2.2}) with the star operator $\star$ of $g|_L$ and $\epsilon|_L$.

In order to show that $F'(0)$ is surjective, we only need consider the case
${\mathcal{X}}={\mathcal{X}}_4$. To this end let us  write each
$$
(\eta,d\zeta)\in C^{1,a}\Omega^1(L)\oplus
dC^{1,a}\Omega^{n-1}(L)
$$
as a convenient form.

 Note that ${\bf
t}(d\omega)=d({\bf t}\omega)$ and ${\bf n}(\delta\omega)=\delta({\bf
n}\omega)$  for any $C^1$-form $\omega$ on $L$ (cf. \cite[Prop.
1.2.6(b)]{Sch}). Since  $C^{0,a}\Omega^{n-1}(L)\subset
L^{2}\Omega^{n-1}(L)$, by \cite[Th.5.7, 5.8]{Mor0} or
\cite[Th.7.7.7, 7.7.8]{Mor} we may write $\zeta= \delta_n\gamma'+
d\gamma''+ h(\zeta)$, where
$$
\gamma'\in C^{1,a}\Omega^{n}_{\bf N}(L)),\quad\gamma''\in
C^{1,a}\Omega^{n-2}_{\bf D}(L),\quad  h(\zeta)\in {\mathcal
H}C^{0,a}\Omega^{n-1}(L).
$$
Moreover  (\ref{e:4.2}) and ${\bf t}(d\omega)=d({\bf t}\omega)$
imply
$$
\delta_n=(-1)^{n(n+1)+1}\star d_0\star:C^{2,a}\Omega^{n}_{\bf
N}(L)\to C^{1,a}\Omega^{n-1}_{\bf N}(L).
$$
 We may assume
$$
d\zeta=-d\star du\quad\mbox{ with }\quad u\in
C^{2,a}\Omega^0_{\bf D}(L).
$$
Similarly, we have
$$
\eta=\delta v+ d\beta+ h(\eta),
$$
where
$$
v\in C^{2,a}\Omega^{2}_{\bf N}(L)),\quad \beta\in
C^{2,a}\Omega^{0}_{\bf D}(L),\quad  h(\eta)\in {\mathcal
H}C^{1,a}\Omega^{1}(L)).
$$
By (\ref{e:4.3}), $d\star\delta v=(-1)^2d (d\star v)=0$ and $d(\star
h(\eta))=(-1)\star\delta h(\eta)=0$. We get
\begin{eqnarray*}
(\eta,d\zeta)&=&\bigl(d\beta-du+du+ \delta v+h(\eta),-d\star(du+
\delta v+h(\eta)\bigr)\nonumber\\
&=&(d\chi + \omega,-d\star\omega),
\end{eqnarray*}
 where
 \begin{equation}\label{e:4.6}
 \begin{split}
&\chi:=\beta-u\in C^{2,a}\Omega^{0}_{\bf D}(L), \\
& \omega:=du+ \delta v+h(\eta)\in C^{1,a}\Omega^{1}(L).
\end{split}
\end{equation}
Note that $C^{2,a}\Omega^{0}_{\bf D}(L)=\{f\in C^{2,a}\Omega^{0}(L)\,|\, f|_{\partial L}=0\}$
and that $R_\alpha\in C^{1,a}(\Gamma(\widehat{N}(L)))$.
If we find a $Z\in C^{1,a}(\Gamma(\widehat{N}(L)))$ such that
\begin{equation}\label{e:4.7}
j^*(\iota_{Z}d\alpha)=\omega,
\end{equation}
then $V:=\chi R_\alpha+ Z$ belongs to ${\mathcal{X}}_4$ and satisfies $F'(0)(V)=(\eta, d\zeta)$.

To obtain (\ref{e:4.7}), consider the symplectic vector bundle $(\xi|_L,
d\alpha|_{\xi|_L})$ with a Lagrangian subbundle $TL$. Let
$TL^{\bot_{\hat g}}_\xi$ be the orthogonal complementary bundle of
$TL$ in $\xi|_L$ with respect to $\hat g$. Then $TL^{\bot_{\hat
g}}_\xi=\xi\cap\widehat{N}(L)$. So $\xi|_L=TL\oplus_{\hat
g}(\xi\cap\widehat{N}(L))$.
  Note that
$\omega$ may be viewed as a section of the bundle ${\rm
Hom}(TL,\mathbb{R})$. We may extend it into a section of ${\rm
Hom}(\xi|_L,\mathbb{R})$, $\hat\omega$, by defining
$$
\hat\omega_p(u+v)=\omega_p(u)
$$
for any $p\in L$ and $u+ v\in T_pL\oplus_{\hat g}(\xi\cap\widehat{
N}(L))_p$, where $u\in T_pL$ and $v\in(\xi\cap\widehat{N}(L))_p$.
Note that $\hat\omega\in C^{1,a}(\Gamma({\rm
Hom}(\xi|_L,\mathbb{R})))$. The non-degeneracy of $d\alpha$ on $\xi$
implies that there exists a unique section $Z:L\to\xi|_L$ such that
$$
(d\alpha)_p(Z(p), A)=\hat\omega_p(A)\quad\forall p\in
L\;\hbox{and}\; A\in\xi_p.
$$
Clearly, $Z\in C^{1,a}(\xi|_L)$. Since $\xi|_L=TL\oplus_{\hat
g}(\xi\cap\widehat{N}(L))$ we get a unique decomposition $Z=Z_1+
Z_2$, where $Z_1\in C^{1,a}(\Gamma(TL))$ and $Z_2\in
C^{1,a}(\Gamma(\xi\cap\widehat{N}(L)))$. Obverse that
$$
j^*(\iota_{Z_1}d\alpha)=0.
$$
In fact, for any $p\in L$ and $u\in T_pL$ it holds that
$$
(j^*(\iota_{Z_1}d\alpha))_p(u)=(\iota_{Z_1}d\alpha))_{j(p)}(j_\ast
u)=(d\alpha)_p(Z_1(p), u)=0
$$
since $T_pL$ is a Lagrangian subspace of $(\xi_p, (d\alpha)_p)$.
Hence we get
$$
(d\alpha)_p(Z_2(p), A)=\hat\omega_p(A)\quad\forall p\in
L\;\hbox{and}\; A\in\xi_p.
$$
This implies $j^*(\iota_{Z_2}d\alpha)=\omega$.  In summary we have
proved:
\begin{claim}\label{cl:4.1}
There exists a unique section $Z:L\to \xi|_L\cap \widehat{N}(L)$
such that {\rm(\ref{e:4.7})} is satisfied. Moreover, $Z$ is also of
class $C^{1,a}$. As a consequence the map $F'(0)$ is surjective.
\end{claim}

\noindent{\bf 4.4. Computing  $\ker(F'(0))$}.
As above let ${\mathcal X}$ be one of ${\mathcal X}_i$,
$i=1,2,3,4$. Let $V\in{\mathcal X}$ sit in
$\ker(F'(0))$. As above we may write $V=fR_\alpha+ Y$, where
$f=\alpha(V)\in C^{2,a}(L)$ and $Y\in C^{2,a}(\Gamma(\xi|_L))$.   (\ref{e:4.5})
yields
\begin{eqnarray}
&&df+ j^\ast(\iota_{Y}d\alpha)=0,\label{e:4.8}\\
 &&-d\star j^\ast(\iota_{Y}d\alpha)=0.\label{e:4.9}
\end{eqnarray}
From (\ref{e:4.8}) we get
\begin{eqnarray*}
0=\delta(df+ j^\ast(\iota_{Y}d\alpha))&=&\delta df+
\delta(j^\ast(\iota_{Y}d\alpha))\\
&=&\delta df+ (-1)^{2n+1}\star d\star(j^\ast(\iota_{Y}d\alpha))
=\delta df
\end{eqnarray*}
because of (\ref{e:4.9}). Hence $\triangle f=0$, i.e., $f$ is a harmonic function.

Note that we have a symplectic orthogonal decomposition $\xi|_L=TL\oplus (JTL)$ with
respect to $d\alpha|_\xi$.  $Y$ in (\ref{e:4.8}) has a unique
decomposition $Y=Z_1+ JZ_2$, where $Z_1$ and $Z_2$ are vector fields on $L$.
Since $j^\ast(\iota_{Z_1}d\alpha)=0$, (\ref{e:4.8}) and (\ref{e:4.9}) become
\begin{eqnarray}\label{e:4.10}
df+ j^\ast(\iota_{JZ_2}d\alpha)=0\quad\hbox{and}\quad d\star j^\ast(\iota_{JZ_2}d\alpha)=0,
\end{eqnarray}
respectively. The first equation shows that
 $Z_2$ is uniquely determined by $df$ because the map $\Theta:TL\to T^\ast L$
defined by $\Theta(u)=j^\ast(\iota_{Ju}d\alpha)$ is an isomorphism.
Write this $Z_2$ as $Z_2(df)$. It is linear in $df$.
 Denote by $\hat{P}:T_LM\to\hat{N}(L)$  the (fibrewise) orthogonal projection with respect to the metric $\hat{g}$.
Then $V=\hat{P}V=f\hat{P}R_\alpha+ \hat{P}(JZ_2)$, and thus $Z_1=(id-\hat{P})(fR_\alpha+ JZ_2)$.

For conveniences we write $F_j$ as the restriction of $F$ on ${\mathcal{X}}_j$, $j=1,2,3,4$.

\noindent{\bf 4.7. Proof of (i)}.
 Decompose $V\in\ker(F'_1(0))$ into $V=fR_\alpha+ Z_1+ JZ_2$ as above,
where $Z_i\in\Gamma(TL)$, $i=1,2$. (\ref{e:4.10}) shows
$j^\ast(\iota_{JZ_2}d\alpha)\in\mathcal{H}C^{2,a}\Omega^1(L)$.
Moreover the boundary condition $d\alpha(N,
V|_{\partial L})=0$ implies $d\alpha(N,
JZ_2|_{\partial L})=0$, i.e., $\iota_{N(x)}j^\ast(\iota_{JZ_2}d\alpha)=0\;\forall x\in{\partial L}$.
This means that the $1$-form $j^\ast(\iota_{JZ_2}d\alpha)$ is tangent to the boundary
with respect to the metric $g|_L$.
By the definition we have ${\bf n}j^\ast(\iota_{JZ_2}d\alpha)=0$
  and
so $\Theta(Z_2)=j^\ast(\iota_{JZ_2}d\alpha)\in\mathcal{H}C^{2,a}\Omega_{\bf N}^1(L)$.
Note that $\mathcal{H}C^{2,a}\Omega_{\bf N}^1(L)\cong H^1(L;\mathbb{R})$
by Hodge theorem (cf. \cite[Theorem~2.6.1]{Sch}). We obtain
$$
\dim\{Z_2\,|\, V=fR_\alpha+ Z_1+ JZ_2\in\ker(F'_1(0))\}\le
\dim H^1(L;\mathbb{R}).
$$
But the first equation in (\ref{e:4.10}) implies that $f$ can be determined by
$j^\ast(\iota_{JZ_2}d\alpha)$ (and so $Z_2$) up to a constant. And
$V=\hat{P}V=f\hat{P}R_\alpha+ \hat{P}(JZ_2)$. We deduce
 $\dim\ker(F'_1(0))\le\dim H^1(L;\mathbb{R})+1$. The first claim is proved.

To get the second claim obverse that we have an linear isomorphism
\begin{equation}\label{e:4.11}
\Xi:\Gamma(N(L))\to \Gamma(\widehat{N}(L)),\;V\mapsto V^v,
\end{equation}
given by the decomposition  $V=V^v+V^h$, where $V^v\in\Gamma(\hat{N}(L))$ and $V^h\in\Gamma(TL)$.
It suffices to prove that $\Xi$ maps $\Gamma(N(L))_W$ into $\ker(F'_1(0))$.
Let $V\in\Gamma(N(L))_W$. Then $d\alpha(N, V|_{\partial L})=0$  by Claim~2.6.
It is easy to see that this implies $d\alpha(N, \Xi(V)|_{\partial L})=0$.
By the assumption there exists a small deformation of $L$ through special Lagrangian submanifolds
with boundary confined in $W$, $j_t:L\to M$, where $|t|\ll 1$, such that $V(x)=\frac{d}{dt}|_{t=0}j_t(x)$
for any $x\in L$. Write $V=\Xi(V)+V^h=fR_\alpha+Y$, where $V^h\in\Gamma(TL)$,
$f=\alpha(V)$ and $Y\in\Gamma(\xi|_L)$. Since
$j^{*}_{t}\alpha=0$ and $j^{*}_{t}{\rm Im}\epsilon=0$, as in (\ref{e:4.5}) we obtain
(\ref{e:4.8}) and (\ref{e:4.9}).   Set $Z:=Y-V^h$. It belongs to $\Gamma(\xi|_L)$ since $TL\subset\xi|_L$,
and $\Xi(V)=fR_\alpha+ Z$.  (\ref{e:4.8}) and (\ref{e:4.9}) imply
\begin{eqnarray}\label{e:4.12}
df+ j^\ast(\iota_{Z}d\alpha)=0\quad\hbox{and}\quad
 -d\star j^\ast(\iota_{Z}d\alpha)=0.
\end{eqnarray}
This means that $\Xi(V)\in \ker(F'_1(0))$.

\noindent{\bf 4.6. Proof of (ii)}.  {\it Proof of statement 1}. Let $V\in\ker(F'_2(0))$. Following
the notations in the proof of (i) we have proved
$\Theta(Z_2)=j^\ast(\iota_{JZ_2}d\alpha)\in\mathcal{H}C^{2,a}\Omega_{\bf N}^1(L)$.
Now $df|_{\partial L}=0$ implies
$$
{\bf t}j^\ast(\iota_{JZ_2}d\alpha)=j^\ast(\iota_{JZ_2}d\alpha)|_{\partial L}=-df|_{\partial L}=0
$$
by the first equation in (\ref{e:4.10}). Hence $\Theta(Z_2)=j^\ast(\iota_{JZ_2}d\alpha)\in\mathcal{H}C^{2,a}\Omega_{\bf D}^1(L)$.
By the strong unique continuation theorem of Aronszajn, Krzywicki and Szarski (cf.\cite[Theorem~3.4.4]{Sch})
we have $\mathcal{H}C^{2,a}\Omega_{\bf N}^1(L)\cap\mathcal{H}C^{2,a}\Omega_{\bf D}^1(L)=\{0\}$.
Then $Z_2=0$ and so $V=\hat{P}V=f\hat{P}R_\alpha+ \hat{P}(JZ_2)=f\hat{P}R_\alpha$.
This shows
\begin{eqnarray}\label{e:4.13}
\ker(F'_2(0))\subset \{f \hat{P}R_\alpha|_L\;|\;f\in C^\infty(L),\;\triangle f=0,\; df|_{\partial L}=0\}
\end{eqnarray}
(For simplicity we write $R_\alpha|_L$
 as $R_\alpha$ without confusion occurring below.) Since $\alpha(V)=\alpha(\Xi(V))$
 and  $\Xi\bigl(\Gamma(N(L))_W\bigr)\subset\ker(F'_1(0))$ we deduce
 $$
 \Xi\bigl(\{V\in\Gamma(N(L))_W\,|\,d\bigl(\alpha({V})|_{\partial L}\bigr)=0\}\bigr)\subset
 \ker(F'_2(0)).
 $$
  But $\{f\in C^\infty(L)\,|\,\triangle f=0,\; df|_{\partial L}=0\}\cong\mathbb{R}^l$
by  \cite[Theorem~3.4.6]{Sch}. These lead to
 $$
 \dim\{V\in\Gamma(N(L))_W\,|\,d\bigl(\alpha({V})|_{\partial L}\bigr)=0\}\le\dim\ker(F'_2(0))\le l.
 $$

 If $\partial L$ is connected, we have
$$
\{V\in\Gamma(N(L))_W\,|\,d\bigl(\alpha({V})|_{\partial L}\bigr)=0\}=
\{V\in\Gamma(N(L))_W\,|\,\alpha({V})|_{\partial L}={\rm constant}\}.
$$
 This case can be included in the proof of the following second statement.


 {\it Proof of statement 2}.
 We claim
\begin{eqnarray}\label{e:4.14}
\ker(F'_3(0))=\{c\hat{P}R_\alpha|_L\;|\, c\in\mathbb{R}\},
\end{eqnarray}
 In fact, as before we can write $V\in \ker(F'(0))$
as $V=fR_\alpha+ Z_1+ JZ_2$ for unique $Z_i\in\Gamma(TL)$, $i=1,2$. Since $f=\alpha(V)$ sattisfies
$\triangle f=0$ and $f|_{\partial L}$ is constant,  so is $f$  by the maximum principle. As above we get $Z_2=0$.
Then $V=fR_\alpha+Z_1$ and hence $V=\hat{P}(fR_\alpha)=f\hat{P}R_\alpha$.
This shows $\ker(F'_3(0))\subset\{c\hat{P}R_\alpha|_L\;|\, c\in\mathbb{R}\}$.

Note that $\dim\{c\hat{P}R_\alpha|_L\;|\, c\in\mathbb{R}\}=1$ is one-dimensional since $\hat{P}R_\alpha(p)=R_\alpha(p)\ne 0$ at each $p\in\partial L$ by Remark~\ref{rm:3.5}.
Moreover $\hat{P}R_\alpha|_L=R_\alpha|_L+ Z$ for some $Z\in\Gamma(TL)$.
It is easy to check that $\hat{P}R_\alpha|_L\in{\mathcal X}_3$ and $F'(0)(\hat{P}R_\alpha|_L)=0$ by
(\ref{e:4.12}). (\ref{e:4.14}) follows immediately.

It remains to prove
\begin{equation}\label{e:4.15}
\Xi\bigl(\{V\in\Gamma(N(L))_W\,|\,\alpha({V})|_{\partial L}={\rm constant}\}\bigr)=\ker(F'_3(0)).
\end{equation}
Since $R_\alpha|_L\in\{V\in\Gamma(N(L))_W\,|\,\alpha({V})|_{\partial L}={\rm constant}\}$
and $\Xi(R_\alpha|_L)=\hat{P}R_\alpha|_L$
we derive from (\ref{e:4.14}) that the right side in (\ref{e:4.15}) is contained in the left one.
 To prove the converse inclusion relation, note that
 every $V\in\Gamma(N(L))_W$ satisfies $d\alpha(N, V|_{\partial L})=0$  by Claim~2.6.
Moreover, as in the proof of (i) we can write $V\in\Gamma(N(L))_W$ as $V=fR_\alpha+Y$, where $f=\alpha(V)$ and $Y\in\Gamma(\xi|_L)$ must satisfy (\ref{e:4.8}) and (\ref{e:4.9}). Then $\Xi(V)=fR_\alpha+ Z$ with $Z=Y-V^h$ satisfies (\ref{e:4.12}), $d\alpha(N,\Xi(V)|_{\partial L})=0$ and $\alpha(V)=\alpha(\Xi(V))$.
These show $\Xi(V)\in\ker(F'_3(0))$, and so the desired inclusion.



\noindent{\bf 4.5. Proof of (iii)}. By a contradiction we assume that
there exists a small deformation of $L$ through special Lagrangian submanifolds
with boundary confined in $W$, $f_t:L\to M$, where $|t|\ll 1$, such that $V(x)=\frac{d}{dt}|_{t=0}f_t(x)$
for any $x\in L$.
Decompose $V$ into $V^v+V^h$, where $V^v\in\Gamma(\hat{N}(L))$ and $V^h\in\Gamma(TL)$.
Then $\alpha(V)=\alpha(V^v)$ and $V^v\ne 0$.
Let $V=fR_\alpha+Y$, where $f=\alpha(V)$ and $Y\in\Gamma(\xi|_L)$. As above $f$ and $Y$ satisfy
(\ref{e:4.8}) and (\ref{e:4.9}). Set $Z:=Y-V^h$. It belongs to $\Gamma(\xi|_L)$ since $TL\subset\xi|_L$,
and $V^v=fR_\alpha+ Z$. It follows from (\ref{e:4.8}) and (\ref{e:4.9}) that
$f$ and $Z$ satisfy (\ref{e:4.12}).
This means that $V^v$ belongs to $\ker(F'_4(0))$.
Since $\triangle f=0$ and $f|_{\partial L}=0$ we get $f=0$ and so $V^v=Z=Z_1+ JZ_2$,
where $Z_i\in\Gamma(TL)$, $i=1,2$.  From the first equation in (\ref{e:4.12}) we deduce that
$j^\ast(\iota_{JZ_2}d\alpha)=0$ and hence $Z_2=0$. Then $V^v=Z_1\in\Gamma(TL)$, which
contradicts to $V^v\in\Gamma(\hat{N}(L))\setminus\{0\}$.

Theorem~\ref{th:1.1} is proved. \qed


\newpage

\section{The proof of Theorem~\ref{th:1.2}}

Let $\langle R_{\alpha}\rangle$ denote the real line bundle
generated by $R_\alpha|_L$. Then the normal bundle of $L$ with
respect to the metric $g$, $N(L)$, is equal to $\langle
R_{\alpha}\rangle\oplus_{g} JTL$. For a small section $V:L\to N(L)$,
the exponent map of $g$ yields a map
$$
\exp_V:L\to M,\;x\mapsto \exp_x(V(x)).
$$
Thus there exists a  neighborhood  $\mathscr{V}$ of $0$ in
\begin{eqnarray*}
{\mathcal{Y}}:=\bigl\{{V}\in C^{2,a}(\Gamma(\langle
R_{\alpha}\rangle))\oplus C^{1,a}(\Gamma(JTL))\,:
\;\alpha(V)|_{\partial L}={\rm const}\bigr\}
\end{eqnarray*}
so that the following map is well-defined:
\begin{equation}\label{e:5.1}
\begin{split}
G:\,&\mathscr{V}\to C^{1,a}\Omega^1(L)\oplus C^{0,a}\Omega^{n}(L),\\
   &V\mapsto\bigl(\exp^{*}_{V}\alpha,2\exp^{*}_{V}{\rm Im}\epsilon\bigr).
\end{split}
\end{equation}
It is $C^1$ (\cite{ToVe}), and ${\rm Im}(G)\subseteq
C^{1,a}\Omega^1(L)\oplus dC^{1,a}\Omega^{n-1}(L)$ as above since
$\exp_V$ is homotopic to the inclusion $j:L\hookrightarrow M$ via
$\exp_{tV}$.

 Considering $G$ as a map to $C^{1,a}\Omega^1(L)\oplus
dC^{1,a}\Omega^{n-1}(L)$, and writing $V=JX+ fR_\alpha$, we may get
\begin{equation}\label{e:5.2}
   \begin{split}
     G'(0)(V)=&\frac{d}{dt}(\exp^{*}_{tV}\alpha,2\exp^{*}_{tV}{\rm Im}\epsilon)|_{t=0}\\
        =&(\mathcal{L}_{V}\alpha,2\mathcal{L}_{V}{\rm Im}\epsilon)|_{L}\\
                =&(df+\iota_{JX}d\alpha,-d*\iota_{JX}d\alpha)|_{L}
     \end{split}
\end{equation}
as above. Moreover,  each $(\eta,d\zeta)\in
C^{1,a}\Omega^1(L)\oplus dC^{1,a}\Omega^{n-1}(L)$ may be
written as $(\eta,d\zeta)=(d\chi + \omega,-d\star\omega)$, where
$\chi$ and $\omega$ are as in (\ref{e:4.6}). Take $f=\chi$, and one
easily find $X\in C^{1,a}(\Gamma(TL))$ such that
$j^*(\iota_{JX}d\alpha)=\omega$. Clearly, such a $V=fR_\alpha+ JX$
satisfies $\alpha(V)|_{\partial L}=0$.  Hence $G'(0)$ is surjective.

Assume that $V=fR_\alpha+ JX$ sits in $\ker(G'(0))$. Then $f$ and
$JX$ satisfy
$$
df+ j^\ast(\iota_{JX}d\alpha)=0,\qquad -d\star
j^\ast(\iota_{JX}d\alpha)=0.
$$
It follows that $\triangle f=\delta df=0$. Recall that $f=\alpha(V)$
is equal to a constant $c$ on $\partial L$. By the maximum principle
we get $f\equiv c$, and hence
$$
j^\ast(\iota_{JX}d\alpha)=0.
$$
From this we derive $JX=0$ as above. This prove
$\ker(G'(0))=\{cR_\alpha\,|\, c\in\mathbb{R}\}$. Hence $(0,0)$ is a
regular value of the restriction of $G$ to a small neighborhood
${\mathscr V}_0$ of $0\in {\mathscr V}$, and thus
 the moduli space $\mathfrak{M}(L)$ is a 1-dimensional smooth manifold
by the implicit function theorem.\qed\vspace{2mm}

Since $ \iota_{R_\alpha}\epsilon=0$ and
$\mathcal{L}_{R_\alpha}\epsilon=0$ we have $\psi_t({\rm
Im}\epsilon)={\rm Im}\epsilon\;\forall t$, where $\psi_t$ is the
flow of $R_\alpha$. For special Legendrian embedding (submanifold)
$p:L\to M$ we obtain $p_t^{*}\alpha=0$ and $p_t^{*}{\rm
Im}\epsilon=0$  with $p_t=\psi_t\circ p$ for any $t$. So
 the deformation in Theorem~\ref{th:1.2} is actually given
by the isometries generated by the Reeb vector field. \vspace{2mm}

\noindent{\rm Rremark~5.1.}\quad If we replace ${\mathscr V}$ by
 a  neighborhood  ${\mathscr W}$ of $0$ in
$$
C^{2,a}(\Gamma(\langle R_{\alpha}\rangle))\oplus
C^{1,a}(\Gamma(JTL)),
$$
then  the  map
$$
\widehat G:\,\mathscr{W}\to C^{1,a}(\Lambda^1(L))\oplus
C^{0,a}(\Lambda^{n}(L)),\quad
  V\mapsto\bigl(\exp^{*}_{V}\alpha,2\exp^{*}_{V}{\rm Im}\epsilon\bigr).
$$
 is still $C^1$ and has the image $
{\rm Im}(\widehat G)\subseteq C^{1,a}(\Lambda^1(L))\oplus
dC^{1,a}(\Lambda^{n-1}(L))$. From the above proof it is easy to see
that $\widehat G'(0)$ is surjective. If $V=fR_\alpha+ JX$ belongs to
$\ker(\widehat G'(0))$, we have $\triangle f=0$ as above. But
$\partial L$ is a nonempty closed manifold, by Theorem 3.4.6 of
\cite{Sch} each $b\in C^\infty(\partial L)$ corresponds to a unique
$f\in C^\infty(L)$ satisfying $\triangle f=0$ and $f|_{\partial
L}=b$. It follows that $\ker(\widehat G'(0))$
 must be of infinite dimension.\vspace{2mm}

The corresponding
problems with \cite[Cor.9]{Bu} and \cite[Th.4.8]{ToVe} can also be
considered similarly.

\vspace{2mm}

\noindent{\rm ACKNOWLEDGEMENTS}. The authors are deeply grateful to
 the anonymous referee for some interesting questions,  numerous comments
 and improved suggestions. We would like to thank Dr. Georgios Dimitroglou Rizell
for carefully checking and valuable suggestions on this corrected version.



\bibliographystyle{amsnumber}

\end{document}